\documentclass[a4paper,10pt]{article}

\usepackage{amsmath}
\usepackage{amscd}
\usepackage{amsfonts}
\usepackage{stmaryrd}
\usepackage{mathrsfs}
\usepackage{amsthm}
\usepackage{ dsfont }
\usepackage{latexsym, amssymb,amsthm,amsmath,lscape,epsfig,setspace,tikz,slashed}
\usepackage[retainorgcmds]{IEEEtrantools}
\usetikzlibrary{decorations.markings}
 \usetikzlibrary{arrows}
\usepackage[pdftex]{hyperref}
\usepackage[margin=10pt,font=small,labelfont=bf]{caption}
\usepackage{a4wide}
\usepackage[all,cmtip]{xy}

\newcommand{\mfg}{\mathfrak{g}}

\newcommand{\R}{\mathbb{R}}
\newcommand{\Z}{\mathbb{Z}}

\newcommand{\mcV}{\mathcal{V}}
\newcommand{\T}{\mathbb{T}}

\newcommand{\C}{\mathbb{C}}
\newcommand{\mbP}{\mathbb{P}}
\newcommand{\mcL}{\mathcal{L}}

\newcommand{\J}{\mathcal{J}}

\newcommand{\lb}{\llbracket}
\newcommand{\rb}{\rrbracket}

\newcommand{\mfB}{\mathfrak{B}}
\newcommand{\mfF}{\mathfrak{F}}
\newcommand\partto{\mathrel{\raisebox{0.08em}{\scalebox{0.67}{$\supset$}}\hspace{-.415em}\to}}

\numberwithin{equation}{section}

\theoremstyle{plain}
\newtheorem{thm}{Theorem}[section]
\newtheorem{prop}[thm]{Proposition}

\theoremstyle{definition}
\newtheorem{lem}[thm]{Lemma}
\newtheorem{defn}[thm]{Definition}
\newtheorem{defn/thm}[thm]{Definition/Theorem}
\newtheorem{ex}[thm]{Example}

\theoremstyle{remark}
\newtheorem{rem}[thm]{Remark}

\begin{document}

\title{Blow-ups in generalized complex geometry}

\author{M.Bailey \thanks{{\tt M.A.Bailey@uu.nl}}  , G.R.Cavalcanti \thanks{{\tt G.R.Cavalcanti@uu.nl}} , J.L.van der Leer Dur\'an \thanks{{\tt J.L.vanderLeerDuran@uu.nl}} \\
       Department of Mathematics\\
Utrecht University\\
}
\date{\vspace{-5ex}}

\maketitle

\abstract{}
\noindent
We study blow-ups in generalized complex geometry. To that end we introduce the concept of holomorphic ideal, which allows one to define a blow-up in the category of smooth manifolds. We then investigate which generalized complex submanifolds are suitable for blowing up. Two classes naturally appear; generalized Poisson submanifolds and generalized Poisson transversals, submanifolds which look complex, respectively symplectic in transverse directions. We show that generalized Poisson submanifolds carry a canonical holomorphic ideal and give a necessary and sufficient condition for the corresponding blow-up to be generalized complex. For the generalized Poisson transversals we give a normal form for a neighborhood of the submanifold, and use that to define a generalized complex blow-up, which is up to deformation independent of choices. 
\vskip12pt
\noindent

\tableofcontents

\section{Introduction}

The notion of blowing up was invented by algebraic geometers in the study of birational transformations. Although it is unclear to the authors when and by whom precisely the notion of blowing up was invented, Zariski \cite{MR0006451} introduced it in a modern language and used it to study singularities. This work culminated in results by Abhyankar and Hironaka on resolutions of singularities in all dimensions. Later Hopf \cite{MR0068008} introduced the corresponding notion in the context of complex analytic geometry. Blowing up a submanifold preserves the class of K\"ahler manifolds and it was pointed out by Gromov in  \cite{MR864505} that it can be defined in the symplectic category as well. This was then used by McDuff in \cite{MR772133} to produce examples of simply-connected non-K\"ahlerian symplectic manifolds. 
\newline
\newline
\noindent
The fact that blow-ups exist in both complex and symplectic geometry naturally raises the question whether the same is true in generalized complex geometry, a concept introduced by Hitchin and developed by Gualtieri \cite{MR2811595} and which unifies complex and symplectic structures into one framework. This question was first dealt with in \cite{MR2574746} where it was shown that a blow-up exists for a non-degenerate point of complex type in a generically symplectic $4$--manifold. This was then used to produce new examples of generalized complex structures on the manifolds $m\C\mbP^2\# n\overline{\C\mbP^2}$ for $m$ odd. 
\newline
\newline
In this paper we study blow-ups in generalized complex geometry. The first step is to understand which submanifolds are suitable for blowing up. In the complex and symplectic categories these are the complex, respectively symplectic submanifolds. There are a number of possible ways to define a generalized complex submanifold, and the one which we will use has complex and symplectic submanifolds as special examples. However, for blowing up this notion is too general and we will restrict ourselves to two special subclasses. The first are the \textsl{generalized Poisson submanifolds}, which look complex in transverse directions. Using the normal form theorem of the first author \cite{MR3128977} we prove that these submanifolds come naturally equipped with a special ideal which gives them a holomorphic flavor, and we use that to construct the blow-up as a differentiable manifold. The question of whether this blow-up has a generalized complex structure for which the blow-down map is holomorphic then boils down to the analogous question in the context of holomorphic Poisson geometry. This has been answered by Polishchuk in \cite{MR1465521} and, building on that, we give necessary and sufficient conditions for blowing up a generalized Poisson submanifold. 

The second class of submanifolds are the \textsl{generalized Poisson transversals}. They look symplectic in transverse directions and, as in the symplectic category, to blow them up we first need a normal form for the generalized complex structure in a neighborhood of the submanifold. Such a neighborhood theorem was already constructed in \cite{2013arXiv1306.6055F} in the context of Poisson geometry, and it has a direct counterpart in our setting. We then blow up the submanifold globally. An elegant way to perform this last step uses reduction methods, just as the symplectic blow-up can be performed using symplectic cuts as shown in \cite{MR1338784}. In contrast with the generalized Poisson submanifolds, the blow-up is not canonical but depends on the specific choice of neighborhood as well as the choice of level set for a specific moment map. The latter is analogous to the symplectic area of the exceptional divisor in the context of symplectic blow-ups. Finally, we show that different choices of models for a neighborhood lead to deformation equivalent blow-ups.
\newline
\newline
The paper is organized as follows: In Section \ref{12:09} we briefly review all the necessary ingredients from generalized complex geometry that are needed in the paper. Most of this material is due to \cite{MR2811595} and all statements without explicit references are from there. We then proceed in Section \ref{12:13} to the blow-up procedure. We first define the notion of \textsl{holomorphic ideal} and argue that this is the natural input to define a blow-up procedure in the category of smooth manifolds. Then, in Section \ref{complex case} we introduce generalized Poisson manifolds and explain the extra assumptions that are needed for the blow-up. In Section \ref{12:15} we define generalized Poisson transversals, give a normal form for their neighborhoods and use it to blow them up. Finally, in Section \ref{11:46:45} we discuss other types of generalized complex submanifolds and give a concrete example of one that can not be blown-up.

\section{Generalized Complex Geometry}
\label{12:09}

Let $M$ be a real $2n$-dimensional manifold equipped with a closed real $3$--form $H$. The main idea of generalized geometry is to replace the tangent bundle $TM$ by the bundle $\T M:=TM\oplus T^\ast M$. The latter carries two natural structures, the first being a fiberwise \textsl{natural pairing} 
\begin{align*}
\langle X+\xi,Y+\eta\rangle:=\frac{1}{2}(\xi(Y)+\eta(X)),
\end{align*}
\noindent
which is a non-degenerate metric of of signature $(2n,2n)$. The second is a bracket on its space of sections which replaces the Lie bracket and is called the \textsl{Courant bracket}. It is given by 
\begin{align*}
\lb X+\xi, Y+\eta \rb:= [X,Y] +\mcL_{{}_X}\eta-\iota_{{}_Y}d\xi -\iota_{{}_Y}\iota_{{}_X}H.
\end{align*}
This version of the Courant bracket is not skew-symmetric but does satisfy the Jacobi identity. 
\begin{defn}
A \textsl{generalized complex structure} on $(M,H)$ is a complex structure $\J$ on $\T M$ which is orthogonal with respect to the natural pairing and whose $(+i)$--eigenbundle $L\subset\T M_\C$ is involutive\footnote{A subbundle of $\T M$ is called involutive if its space of sections is closed with respect to the Courant bracket.}.
\end{defn}
\noindent
A Lagrangian, involutive subbundle $L\subset \T M_\C$ is also called a \textsl{Dirac structure}, and it follows from the definition that generalized complex structures correspond in a one to one fashion with Dirac structures $L$ satisfying the non-degeneracy condition $L\cap \bar{L}=0$. 

\begin{ex}
\label{11:37:30}
The main examples are provided by complex and symplectic geometry: 
\begin{align}
\J_I=
\begin{pmatrix}
-I & 0 \\
0 & I^\ast
\end{pmatrix},
\ \ \ \ \ 
\J_\omega=
\begin{pmatrix}
0 & -\omega^{-1} \\
\omega & 0
\end{pmatrix},
\label{13:31}
\end{align}
with associated Dirac structures $L_I=T^{0,1}\oplus (T^\ast)^{1,0}$ and $L_\omega=(1-i\omega)T$. Another important example is provided by a holomorphic Poisson structure $(I,\sigma)$. If $\sigma=Q-iIQ$ then 
\begin{align}
\J_\sigma=
\begin{pmatrix}
-I & 4IQ \\
0 & I^\ast
\end{pmatrix}
\label{13:30}
\end{align}
and $L_{I,\sigma}=T^{0,1}\oplus (1+\sigma)(T^\ast)^{1,0}$. In these examples the $3$--form is taken to be $0$.
\end{ex}
\noindent
A useful way to look at generalized complex structures is through spinors. There is a natural action of the Clifford algebra of $(\T M,\langle, \rangle)$ on differential forms given by 
\begin{align*}
(X+\xi)\cdot \rho=\iota_{{}_X}\rho+\xi\wedge \rho,
\end{align*} 
yielding an identification between the space of differential forms and the space of spinors for $Cl(\T M,\langle,\rangle)$. 
A line subbundle $K\subset \Lambda^\bullet T^\ast M_\C$ gives rise to an isotropic subbundle $L\subset \T M_\C$ by taking its annihilator
\[
L=\{X+\xi\in \T M_\C|(X+\xi)\cdot K=0\}.
\]
This gives rise to a one-to-one correspondence between Dirac structures $L\subset \T M_\C$ and complex line bundles $K\subset \Lambda^\bullet T^\ast M_\C$ which satisfy the following two conditions. Firstly, $K$ has to be generated by \textsl{pure spinors}, i.e. forms $\rho$ which at each point $x$ admit a decomposition
\begin{align}
\rho_x=e^{B+i\omega}\wedge \Omega
\label{13:09:56}
\end{align}
where $B+i\omega$ is a $2$--form and $\Omega$ is decomposable. This condition is equivalent to $L$ being of maximal rank. Secondly, if $\rho$ is a local section of $K$ there should exist $X+\xi\in \Gamma(\T M_\C)$ with
\[d^H\rho=(X+\xi)\cdot \rho.	\]
This condition amounts to the involutivity of $L$. The condition $L\cap \bar{L}=0$ can then be expressed in spinor language using the \textsl{Chevalley pairing}: If $\rho\in \Gamma(K)$ is non-vanishing then
\begin{align*}
L\cap \bar{L}=0 \Longleftrightarrow (\rho,\bar{\rho})_{{}_{Ch}}:=(\rho\wedge \bar{\rho}^T)_{top}\neq 0.
\end{align*}
The superscript $T$ stands for transposition, acting on a degree $l$--form by $(\beta_1 \ldots \beta_l)^T=\beta_l\ldots\beta_1$, and the subscript $top$ stands for the highest degree component. If $\rho$ is given by (\ref{13:09:56}) at a particular point $x$ then this condition becomes 
\begin{align}
\omega^{n-k}\wedge \Omega\wedge \bar{\Omega}\neq 0,
\label{14:35:04}
\end{align}
 where $2n$ is the real dimension of $M$ and $k=deg(\Omega)$. The line bundle $K$ associated to a generalized complex structure $\J$ is called the \textsl{canonical line bundle}, and the integer $k$ appearing in (\ref{14:35:04}) is called the \textsl{type} of $\J$ at $x$. Structures of type $0$ are called symplectic and those of maximal type $n$ complex\footnote{The reason being that they are equivalent to symplectic or complex structures where equivalence is defined in Definition \ref{12:24}.}. Another description of the type is as follows. Every generalized complex structure naturally induces a Poisson structure given by the composition 
\begin{align}
\pi_\J:T^\ast M\hookrightarrow \T M \xrightarrow{\J} \T M \twoheadrightarrow TM.
\label{12:20}
\end{align}
The conormal bundle to the leaves, i.e. the kernel of $\pi_\J$, is given by the complex distribution 
\begin{align*}
\nu_\J:=T^\ast M \cap \J T^\ast M.
\end{align*}
Note that $\nu_\J$ might be singular as its complex dimension can jump in even steps from one point to the next. 
The type at a point $x$ is then given by
\begin{align*}
type_x(\J)=dim_\C(\nu_\J)_x=\frac{1}{2}corank_\R(\pi_\J)_x.
\end{align*}
\noindent
Having laid out the relevant geometric structures we need to define morphisms between them.
\begin{defn}
A \textsl{generalized map} between $(M_1,H_1)$ and $(M_2,H_2)$ is a pair $\Phi:=(\varphi,B)$, where $\varphi:M_1\rightarrow M_2$ is a smooth map and $B\in \Omega^2(M_1)$ satisfies $\varphi^\ast H_2=H_1+dB$. 
\end{defn}
\noindent
We will often abbreviate $(\varphi,0)$ by $\varphi$ and drop the prefix ``generalized". An important role is played by \textsl{$B$-field transformations}, maps of the form\footnote{The minus sign is chosen so that $e^B\wedge ((X+\xi)\cdot \rho)=(e^B_\ast(X+\xi))\cdot e^B\wedge\rho$.} $(Id,-B)=:e^B_\ast$. They act on $\T M$ via
\begin{align}
e^B_\ast:X+\xi\mapsto X+\xi-\iota_XB.
\label{09:25:06}
\end{align}
Given $u\in \Gamma(\T M)$ we denote by $ad(u):\Gamma(\T M)\rightarrow \Gamma(\T M)$ the adjoint action with respect to the Courant bracket. This infinitesimal symmetry has a flow, i.e. a family of isomorphisms $\psi_t:\T M\rightarrow \T M$ with
$
d/dt (\psi_t(v))=-\lb u ,\psi_t(v) \rb.
$
If $u=X+\xi$ and $\varphi_t$ is the flow of $X$ then
\begin{align}
\psi_t=
(\varphi_t)_\ast \circ e_\ast^{-\int_0^t\varphi_r^\ast(d\xi+\iota_{{}_X}H)dr}.
\label{10:12:54}
\end{align}
A map $\Phi=(\varphi,B)$ gives rise to a correspondence:
\begin{align*}
 X +\xi \underset{\Phi}{\sim} Y+\eta \overset{\text{def}}{\Longleftrightarrow} \varphi_\ast X=Y, \ \ \xi=\varphi^\ast \eta -\iota_{{}_X}B.
\end{align*}
\begin{defn}
A map $\Phi:(M_1,H_1,\J_1)\rightarrow(M_2,H_2,\J_2)$ is called \textsl{generalized holomorphic} if 
\begin{align*}
X +\xi \underset{\Phi}{\sim} Y+\eta \Longrightarrow \J_1(X +\xi) \underset{\Phi}{\sim} \J_2(Y+\eta).
\end{align*}
It is called an \textsl{isomorphism} if it is in addition invertible. 
\label{12:24}
\end{defn}
\begin{rem}
It follows immediately from the definition that $\varphi$ is a Poisson map, i.e. $\varphi_\ast \pi_{\J_1}=\pi_{\J_2}$. This is quite restrictive, for example if the target is symplectic then $\varphi$ has to be a submersion. In the complex category we recover the usual notion of holomorphic maps.
\label{13:34}
\end{rem} 
\noindent
In case $\varphi$ is a diffeomorphism a more concrete description in terms of spinors can be given. If $K_i$ is the canonical bundle for $\J_i$, $\Phi$ being an isomorphism amounts to
\begin{align*}
K_1=e^B\wedge \varphi^\ast K_2.
\end{align*}
\noindent
We now state the analogue of the Newlander-Nirenberg and Darboux theorems in generalized complex geometry.
\begin{thm} (\cite{MR3128977}) Let $(M,H,\J)$ be a generalized complex manifold. If $x\in M$ is a point where $\J$ has type $k$ then a neighborhood of $x$ is isomorphic to a neighborhood of $(0,0)$ in 
\begin{align}
(\R^{2n-2k},\omega_{st}) \times (\C^k,\sigma)
\label{13:19}
\end{align} 
where $\omega_{st}$ is the standard symplectic form, $\sigma$ is a holomorphic Poisson structure which vanishes at $0$ and the $3$--form is zero. 
\label{16:12:53}
\end{thm}
\noindent
Finally we come to the notion of a generalized complex submanifold. For this the notion of holomorphic map as defined above is actually too restrictive. Let $\Phi=(\varphi,B)$ be a map and $L_2$ a Dirac structure on $(M_2,H_2)$. We define the \textsl{backward image} of $L_2$ along $\Phi$ by
\begin{align}
\mfB \Phi(L_2):=\{X+\varphi^\ast\xi-\iota_{{}_X}B |\varphi_\ast X+\xi\in L_2\}\subset \T M_1.
\end{align}
This is a Dirac structure on $(M_1,H_1)$, provided it is a smooth vector bundle. A sufficient condition for that is that $ker(d\varphi^\ast)\cap \varphi^\ast L$ is of constant rank. Similarly, the \textsl{forward image} of a Dirac structure $L_1$ on $(M_1,H_1)$ is given by 
\begin{align*}
\mfF\Phi(L_1):=\{ \varphi_\ast X+\xi | X+\varphi^\ast \xi -\iota_XB\in L_1		\}\subset \varphi^\ast \T M_2.
\end{align*}
This will be smooth if $ker(\varphi_\ast)\cap e^{-B}_\ast L$ has constant rank, and projects down to $M_2$ if it is constant along the fibers of $\varphi$. In case $\varphi$ is a diffeomorphism we have $\mfF\Phi(L)=\varphi_\ast (e^{-B}_\ast (L))$. More information on this, including proofs, can be found e.g. in \cite{MR3098084}. 
\begin{defn}
A \textsl{generalized complex submanifold} is a submanifold $i:Y\hookrightarrow (M,H,\J)$ such that $\mfB i(L)$ is generalized complex, i.e. is smooth and satisfies $\mfB i(L)\cap \overline{\mfB i(L)}=0$. 
\label{20:22}
\end{defn}
\begin{rem}
A sufficient condition for smoothness is that $N^\ast Y\cap \J N^\ast Y$ is of constant rank. Moreover, the second condition is equivalent to $\J N^\ast Y \cap (N^\ast Y)^\perp \subset N^\ast Y$. In complex or symplectic manifolds we recover the usual notion of complex, respectively symplectic submanifolds. Also, a point is always a generalized complex submanifold. Note that in the symplectic case the inclusion map is only generalized holomorphic if $Y$ is an open subset.  
\end{rem}


\section{Blowing up submanifolds}
\label{12:13}  

Before going to generalized complex geometry we discuss the notion of blowing up submanifolds in a more abstract setting. 
\begin{defn}
Let $M$ be a real manifold and let $C^\infty(M)$ be the space of complex valued smooth functions. Let $Y\subset M$ be a closed submanifold of real codimension $2l$ with $l\geq 1$. A \textsl{holomorphic ideal} for $Y$ is an ideal $I_Y\subset C^\infty(M)$ with the following properties:  
\begin{itemize}
\item[i)] $I_Y|_{{}_{M\backslash Y}}=C^\infty(M)|_{{}_{M\backslash Y}}$.
\item[ii)] Each $y\in Y$ has a neighborhood $U$ together with $z^1,\ldots ,z^l\in I_Y(U)$ such that \newline
$z:=(z^1,\ldots,z^l):U\rightarrow \C^l$ is a submersion with $Y\cap U=z^{-1}(0)$ and $I_Y|_{{}_U}=\langle z^1,\ldots ,z^l\rangle$.
\end{itemize}
\label{10:08}
\end{defn} 
\noindent Basically $I_Y$ is a choice of ideal which has $Y$ as its zero set but makes it look complex in transverse directions. A holomorphic ideal turns $NY$ into a complex vector bundle via $N^\ast Y_\C=N^{\ast 1,0} Y \oplus N^{\ast 0,1}Y$, where $N_y^{\ast 1,0} Y:=\langle d_yz   | \ z\in I_Y\rangle$. Given a complex structure on $NY$ there are many holomorphic ideals inducing it, and one way to obtain them is as follows. The zero section in $NY$ carries a natural holomorphic ideal generated by $\Gamma(N^{\ast 1,0}Y)$, viewed as fiberwise linear functions on $NY$. Using a tubular embedding of $NY$ into $M$ we can then glue this ideal on $NY$ to the trivial ideal on $M\backslash Y$. 

We will mainly be interested in holomorphic ideals for smooth submanifolds but in order to state the definition of the blow-up we need to consider also singular submanifolds of codimension $1$. 
\begin{defn}
A \textsl{divisior} on $M$ is an ideal $I_Y \subset C^\infty(M)$ which locally can be generated by a single function and whose zero set $Y$ is nowhere dense in $M$. 
\end{defn}
\noindent Equipped with these definitions we can define the notion of blowing up in the same way as is usually done in algebraic geometry.
\begin{defn}
Let $Y\subset M$ be a closed submanifold and $I_Y$ a holomorphic ideal for $Y$. The \textsl{blow-up} of $I_Y$ in $M$ is defined as a smooth manifold $\widetilde{M}$ together with a smooth \textsl{blow-down} map $p:\widetilde{M}\rightarrow M$ such that $I_{\widetilde{Y}}:=p^\ast I_Y$ is a divisor, and which is universal in the following sense: For any smooth map $f:X\rightarrow M$ such that $f^\ast I_Y$ is a divisor, there is a unique $\widetilde{f}:X\rightarrow \widetilde{M}$ such that the following diagram commutes: 
\begin{align*}
\xymatrix{
X \ar[rd]_f \ar@{-->}[r]^{\widetilde{f}}
&\widetilde{M} \ar[d]^p
\\
&M}
\end{align*}       
\end{defn}

\begin{thm}
The blow-up $(\widetilde{M},p)$ exists and is unique up to unique isomorphism. Moreover, $p:\widetilde{M}\backslash \widetilde{Y}\rightarrow M\backslash Y$ is a diffeomorphism, $I_{\widetilde{Y}}$ is smooth and $p:\widetilde{Y}\rightarrow Y$ is isomorphic to $\mbP(NY)\rightarrow Y$.  
\end{thm}

\begin{proof}
By definition we can cover $M$ by charts which are either disjoint from $Y$ or are of the form $\C^l\times \R^m$ with coordinates $(z^1,\ldots,z^l,x^1,\ldots ,x^m)$, where the $z^i$ are as in Definition \ref{10:08} $ii)$ and $x^i$ are coordinates on $Y$. If we can construct the blow-up on each individual chart then the universal property implies that all the local constructions can be glued into the desired manifold $\widetilde{M}$. On a chart not intersecting $Y$ we do nothing as $I_Y$ is already (trivially) a divisor there. On a chart $U=\C^l\times \R^m$ as above with $Y\cap U=\{0\}\times \R^m$  we define $\widetilde{U}:=\widetilde{\C^l}\times \R^m$ and $p=(\pi,Id):\widetilde{U}\rightarrow U$ where $\pi:\widetilde{\C^l}\rightarrow \C^l$ is the blow-up of the origin. Recall that 
\begin{align*}
\widetilde{\C^l}=\{(z,[x])|z\in [x]\}\subset \C^l\times \mbP^{l-1}
\end{align*}
has a cover by $l$ charts on which $\pi$ has the form
\begin{align}
(v^1,\ldots, v^{i-1},z^i,v^{i+1},\ldots,v^l)\mapsto (z^iv^1,\ldots, z^iv^{i-1},z^i, z^iv^{i+1},\ldots ,z^iv^l)
\label{11:10}
\end{align}    
for $i\leq l$. Now suppose that $f:X\rightarrow U$ is a map such that $f^\ast (I_Y|_{{}_U})$ is a divisor with nowhere dense zero set $D$. The desired lift $\widetilde{f}:X\rightarrow U$ is already uniquely defined on $X\backslash D$ because $\pi$ is an isomorphism over $\C^l\backslash \{0\}$, so we only have to show that $\widetilde{f}$ extends smoothly over $D$. To that end write $f=(f^1,\ldots, f^l, f'^1,\ldots, f'^m)$, so that $f^\ast (I_Y|_{{}_U})=\langle f^1,\ldots, f^l\rangle$. By definition of being a divisor there exists, on a neighborhood $V$ of any $x_0\in D$, a function $g$ with $\langle g\rangle=\langle f^1,\ldots, f^l\rangle$. Therefore there exist $a^i,b_i\in C^\infty(V)$ with $f^i=a^ig$ and $g=\sum_i b_if^i$ and so, since $g\neq 0$ on a dense set, we obtain $\sum_ia^ib_i=1$. In particular there is an index $i_0$ such that, after possibly shrinking $V$, $a^{i_0}$ is nowhere zero. The map $\widetilde{f}:V\backslash D\rightarrow \widetilde{U}$ maps into the chart (\ref{11:10}) for $i=i_0$, where it is necessarily of the form 
\begin{align*}
\widetilde{f}:x\mapsto \big(\frac{f^1(x)}{f^{i_0}(x)},\ldots, f^{i_0}(x), \ldots, \frac{f^l(x)}{f^{i_0}(x)},f'^1(x),\ldots,f'^m(x)\big).
\end{align*}
Since $f^i/f^{i_0}=a^i/a^{i_0}$ we see that $\widetilde{f}$ indeed extends smoothly over the whole of $V$, and therefore over the whole of $D$. So, from the above discussion the blow-up $p:\widetilde{M}\rightarrow M$ indeed exists and is unique. Its further mentioned properties are easily verified from the construction.
\end{proof}
\begin{rem}
It follows from the universal property that the blow-up construction is functorial, i.e. for any map $f:(M_1,I_{Y_1})\rightarrow (M_2,I_{Y_2})$ with $f^\ast I_{Y_2}=I_{Y_1}$, there is a unique map $\widetilde{f}:\widetilde{M_1}\rightarrow \widetilde{M_2}$ making the obvious diagram commute. Note that $f^\ast I_{Y_2}=I_{Y_1}$ implies that the induced map $df:NY_1\rightarrow NY_2$ is complex linear and injective. One case where this occurs is when $f:M_1\rightarrow (M_2,I_{Y_2})$ is transverse to $Y_2$. Then $f^\ast I_{Y_2}$ is a holomorphic ideal for $Y_1:=f^{-1}Y_2$.  
\end{rem}


\subsection{Generalized Poisson submanifolds}
\label{complex case}

In this section we will look at generalized complex submanifolds which are complex in transverse directions. The precise definition is as follows. 

\begin{defn}
Let $\J$ be a generalized complex structure on $M$. A \textsl{generalized Poisson submanifold} is a submanifold $Y\subset M$ such that $\J N^\ast Y=N^\ast Y$. 
\end{defn} 
\noindent
This condition is equivalent to $\J N^\ast Y\cap (N^\ast Y)^\perp =N^\ast Y$, hence generalized Poisson submanifolds are automatically generalized complex\footnote{Note that $N^\ast Y_\C\cap i^\ast L=N^\ast Y_\C$ so $\mfB i(L)$  is automatically smooth, where $i:Y\hookrightarrow M$ denotes the inclusion.} in the sense of Definition \ref{20:22}. Since $\J$ is orthogonal it also preserves $(N^\ast Y)^\perp$ and this gives an explicit description of the generalized complex structure induced on $Y$ via $\T Y\cong (N^\ast Y)^\perp/N^\ast Y$. In this description it is clear that the inclusion map is generalized holomorphic and so $Y$ is a Poisson submanifold for $\pi_{\J}$, justifying the terminology. The key fact in the blow-up theory of generalized Poisson submanifolds is the following.
\begin{prop}
\label{11:14:11}
Let $Y\subset (M,\J)$ be a closed generalized Poisson submanifold. There is a canonical holomorphic ideal $I_Y$ whose associated complex structure on $N^\ast Y$ is given by $\J$. 
\end{prop}
\begin{proof}
Consider a generalized complex chart $U=(\R^{2n-2k},\omega_{st}) \times (\C^k,\sigma)$ around a point in $Y$ as provided by Theorem \ref{16:12:53}. Since $Y$ is a union of symplectic leaves we have $Y\cap U=W\times Z$ where $W\subset \R^{2n-2k}$ is open and $Z\subset \C^k$ is a complex submanifold which is Poisson for $\sigma$. By choosing appropriate holomorphic coordinates $z^i$ on $\C^k$ we may assume that $Z=\{z^1,\ldots, z^l=0\}$ and a natural choice of holomorphic ideal for $Y$ in $U$ is then given by $\langle z^1,\ldots, z^l\rangle$. To patch these local ideals into a global one we need to show that on the overlap of two charts the corresponding ideals match. So suppose $(\R^{2n-2k},\omega_i)\times (\C^k,\sigma_i)$, $i=1,2$, are two local models\footnote{Strictly speaking we should look at open neighborhoods of $0$ but for sake of notation we suppress this. Also note that we can assume that the ``$k$" in both charts is the same, as the type can only jump in even steps and $(\R^{4s},\omega_{st})$ is isomorphic to $(\C^{2s},\sigma_0)$ for $\sigma_0$ an invertible holomorphic Poisson structure.} and suppose that $(\varphi,B)$ is a generalized complex isomorphism between them which maps $Y$ to itself. Let $(x,z)$ and $(y,w)$ be coordinates on the two charts, where $x,y$ and $z,w$ denote the symplectic, respectively complex directions, and such that $I_Y$ is given by $\langle z^1,\ldots,z^l\rangle$, respectively $\langle w^1,\ldots,w^l\rangle$. By symmetry it suffices to  show that $\varphi^\ast w^i\in\langle z^1,\ldots ,z^l\rangle$ for all $i\leq l$. As is shown in \cite[Ch.VI]{MR0212575}, this condition may be verified on the level of Taylor series and since  $\varphi^\ast w^i\in \langle z^1,\ldots ,z^l,\bar{z}^1,\ldots ,\bar{z}^l \rangle$ because $\varphi(Y)=Y$, we only need to verify that 
\begin{align}
\left. \frac{\partial^r w^i}{\partial \bar{z}^{i_1}\ldots \partial \bar{z}^{i_r}}\right |_{Y}=0, \ \ \ \forall r\geq 0, \ \forall i,i_1,\ldots,i_r\in \{1,\dots ,l\}.
\label{eqn4}
\end{align}
Here we are abbreviating $w^i:=\varphi^\ast w^i$. The case $r=0$ reads $w^i|_{{}_Y}=0$, which is satisfied since $\varphi(Y)=Y$. To verify (\ref{eqn4}) we first write out what it means for $(\varphi,B)$ to be an isomorphism:
\begin{align}
e^{i\omega_1}\wedge e^{\sigma_1}(dz^1\ldots dz^k)=e^{f+B+i\omega_2}\wedge e^{\sigma_2}(dw^1\ldots  dw^k).
\label{eqn1}
\end{align}
The factor $e^f$ is there because we are taking representatives of the spinor line. At $Y$, using that $Y$ is Poisson, (\ref{eqn1}) becomes 
\begin{align*}
e^{i\omega_1}\wedge dz^1\ldots dz^l\wedge e^{\sigma_1}(dz^{l+1} \ldots  dz^k)=&e^{f+B+i\omega_2}\wedge dw^1 \ldots  dw^l\wedge e^{\sigma_2}(dw^{l+1} \ldots dw^k).
\end{align*}
\noindent
Now apply $dw^i\wedge \iota_{\partial_{\bar{z}^{i_1}}}$, with $i,i_1\leq l$, to both sides. The left hand side vanishes while the only survivor on the right is given by
\begin{align*}
 \frac{\partial w^i}{\partial \bar{z}^{i_1}}e^{f+B+i\omega_2}\wedge dw^1\ldots dw^l\wedge e^{\sigma_2}(dw^{l+1} \ldots dw^k),
\end{align*}
so (\ref{eqn4}) holds for $r\leq 1$. This implies in particular that the forms $dz^1\ldots dz^l$ and $dw^1\ldots dw^l$ are proportional along $Y$, where again we think of $w^i$ as a function of $(x,z)$. 

Suppose inductively that for some $m\geq 1$ Equation (\ref{eqn4}) is satisfied for all $r\leq m$. Apply $dw^i \wedge\mcL_{\partial_{\bar{z}^{i_1}}}\ldots\mcL_{\partial_{\bar{z}^{i_{m}}}}$, for any $i,i_1,\ldots, i_m\leq l$, to both sides of (\ref{eqn1}) and evaluate the resulting expression at $Y$. The left hand side will vanish again because $\omega_1$ is independent of $z$ and $\sigma_1$ is holomorphic. Using multi-index notation, the Leibniz rule gives
\begin{align}
0=dw^i\wedge \sum_{\substack{I\sqcup J\sqcup K\sqcup L= \\ \{i_1,\ldots, i_m\} }} \mcL_{\partial_{\bar{z}^I}}(e^{f+B+i\omega_2}) \mcL_{\partial_{\bar{z}^J}}(e^{\sigma_2}) \mcL_{\partial_{\bar{z}^K}}(dw^1\ldots dw^l)\mcL_{\partial_{\bar{z}^L}}(dw^{l+1}\ldots dw^k).
\label{eqn3}
\end{align}
\begin{itemize}
\item[] \textbf{Claim:} We have $\mcL_{\partial_{\bar{z}^J}}\sigma_2(dw^j)|_{{}_Y}=0$ for all $J\subset \{i_1,\ldots, i_m\}$ and $j\leq l$.
\end{itemize}
Let us accept this claim for the moment and continue with the proof. 
We compute
\begin{align}
\mcL_{\partial_{\bar{z}^K}}dw^j 
=& \sum_{1 \leq a \leq k}\frac{\partial^{|K|+1} w^j}{\partial z^a\partial\bar{z}^K}dz^a+\sum_{1 \leq a \leq k}\frac{\partial^{|K|+1} w^j}{\partial \bar{z}^a\partial\bar{z}^K}d\bar{z}^a+\sum_{1\leq b \leq 2n-2k}\frac{\partial^{|K|+1} w^j}{\partial x^b\partial\bar{z}^K}dx^b.
\label{eqn2}
\end{align}
If $j\leq l$, the function $\partial^{|K|}w^j/\partial \bar{z}^{|K|}$ vanishes along $Y$ by the induction hypothesis. Hence, at $Y$ the first and second terms above with $a> l$ together with the entire third term vanish, because we differentiate in directions tangent to $Y$. If in addition $|K| < m$, the second term vanishes all together by the induction hypothesis. It follows that for $K\subsetneq \{i_1,\ldots, i_m\}$, $\mcL_{\partial_{\bar{z}^K}}(dw^1\ldots dw^l)|_{{}_Y}$ is proportional to $(dw^1\ldots dw^l)|_{{}_Y}$. Using the Claim, these terms all disappear from (\ref{eqn3}) because we wedge everything with $dw^i$. It is then readily verified that (\ref{eqn3}) reduces to 
\begin{align*}
0=e^{f+B+i\omega_1}e^{\sigma_2}\sum_{1\leq i_{m+1}\leq l} \frac{\partial^{m+1} w^i}{\partial \bar{z}^{i_1}\ldots \partial \bar{z}^{i_{m+1}}}		d\bar{z}^{i_{m+1}} dw^1\ldots dw^k
\end{align*} 
at $Y$. So (\ref{eqn4}) holds for $r=m+1$ as well and therefore for all $r$ by induction.
\end{proof}

\noindent \textsl{Proof of Claim.} 
If we write $\sigma_2=\sigma_2^{ab}\partial_{w^a}\partial_{w^b}$, the Poisson condition implies that $\sigma_2^{ab}$ vanishes at $Y$ for $a\leq l$ or $b\leq l$. A repeated Lie derivative on $\sigma_2$ will be a sum of terms of the form
\begin{align}
 \frac{\partial^r \sigma_2^{ab}}{\partial \bar{z}^{i_1}\ldots \partial \bar{z}^{i_r}}(\mcL_{\partial_{\bar{z}^{j_1}}}\ldots \mcL_{\partial_{\bar{z}^{j_s}}}\partial_{w^a})(\mcL_{\partial_{\bar{z}^{k_1}}}\ldots \mcL_{\partial_{\bar{z}^{k_t}}}\partial_{w^b}).
\label{13:07:23}
\end{align}
Using the chain rule and the fact that $\sigma_2$ is holomorphic we can rewrite the first term in terms of $w$--derivatives. By the induction hypothesis there are no derivatives in the $w^i$--directions for $i\leq l$, because these come together with a term of the form $\partial w^i/\partial \bar{z}^{i_j}$ or a further derivative thereof. Moreover, if either $a\leq l$ or $b\leq l$ there are also no $w^i$--derivatives for $i>l$ because these are tangent to $Y$ along which $\sigma^{ab}$ is constantly equal to zero. Hence (\ref{13:07:23}) will only be nonzero at $Y$ for $a,b>l$ and so to prove the Claim it suffices to show that $(\mcL_{\partial_{\bar{z}^{j_1}}}\ldots \mcL_{\partial_{\bar{z}^{j_s}}}\partial_{w^a})(dw^j)|_{{}_Y}=0$ for $a>l, j\leq l$. Abbreviating $J=\{j_1,\ldots, j_s\}$ we have  
\begin{align*}
0=\mcL_{\bar{z}^J}(dw^j(\partial_{w^a}))=\sum_{J_1\sqcup J_2=J} (\mcL_{\bar{z}^{J_1}}dw^j) (\mcL_{\bar{z}^{J_2}}\partial_{w^a}).
\end{align*}
From Equation (\ref{eqn2}) and the comments below it we see that $\mcL_{\bar{z}^{J_1}}dw^j$ is a linear combination of $dw^{j'}$ with $j'\leq l$. The result then follows by induction over $s$. \qed
\newline
\newline
\noindent
Having a canonical holomorphic ideal for $Y$ we obtain a canonical blow-up $\widetilde{M}$. We now investigate whether $\widetilde{M}$ carries a generalized complex structure for which the blow-down map $p$ is holomorphic. Clearly this structure exists and is unique on $\widetilde{M}\backslash \widetilde{Y}$ and we only need to verify whether it extends over $\widetilde{Y}$. From the definition of the ideal $I_Y$ and the blow-up construction, $p$ is locally given by 
\begin{align*}
\R^{2(n-k)} \times Bl_{{}_Z} \C^k\rightarrow \R^{2(n-k)} \times \C^k
\end{align*}  
where $Bl_{{}_Z}\C^k$ is the complex blow-up of $Z\subset \C^k$. The target is equipped with the generalized complex structure determined by the standard symplectic form on $\R^{2(n-k)}$ and a holomorphic Poisson structure $\sigma$ on $\C^k$. Clearly this structure lifts if and only if $\sigma$ lifts. So we are led to the following question: When does a holomorphic Poisson structure lift to a blow-up? This was addressed by Polishchuk in \cite{MR1465521} and for completeness we review the results here in a more differential geometric language. Recall that $Z\subset (X,\sigma)$ is a holomorphic Poisson submanifold if and only if its holomorphic ideal sheaf $I_Z$ of functions vanishing on $Z$ is a Poisson ideal. In that case $N^{\ast 1,0} Z$ inherits a fiberwise Lie algebra structure, given by the Poisson bracket under the natural isomorphism $N^{\ast 1,0} Z\cong I_Z/I_Z^2$. To state the blow-up conditions on $Z$ we need the following terminology.
\begin{defn}
A Lie algebra $\mfg$ is \textsl{degenerate} if the map $\Lambda^3\mfg\rightarrow Sym^2(\mfg)$ given by
\begin{align*}
x\wedge y \wedge z\mapsto [x,y]z+[y,z]x+[z,x]y
\end{align*}
vanishes. 
\end{defn}
\begin{rem} As stated this definition depends on the field over which $\mfg$ is defined. If $\mfg$ is real then $\mfg$ is degenerate over $\R$ if and only $\mfg_\C$ is so over $\C$. However, if $\mfg$ is complex we can also consider it over $\R$ by forgetting the complex structure and then degeneracy over $\R$ implies degeneracy over $\C$ but not vice versa. It is shown in \cite{MR1465521} that degeneracy is equivalent to being either Abelian or isomorphic to the algebra generated by $e_1,\ldots, e_{n-1},f$, with relations $[e_i,e_j]=0$ and $[f,e_i]=e_i$. Note that $2$--dimensional Lie algebras are always degenerate. 
\label{20:16}
\end{rem}
\
\newline
\noindent
If $Z$ is Poisson we call $N^{\ast 1,0} Z$ degenerate if its fiberwise Lie algebra structure is degenerate over $\C$. This is equivalent to the condition
\begin{align}
\label{14:07:39}
\{f,g\}h+\{g,h\}f+\{h,f\}g\in I_Z^3 \ \ \ \ \forall f,g,h\in I_Z.
\end{align}
Now let $p: \widetilde{X}\rightarrow X$ denote the complex blow-up along a complex submanifold $Z$, and let $\widetilde{Z}$ be the exceptional divisor. We say that $\sigma$ can be lifted if there exists a holomorphic Poisson structure $\widetilde{\sigma}$ on $\widetilde{X}$ for which $p$ is a Poisson map. Note that a lift is necessarily unique, because $p$ is an isomorphism almost everywhere. 
\begin{prop} (\cite{MR1465521})
There exists a lift $\widetilde{\sigma}$ on $\widetilde{X}$ if and only if $Z$ is a Poisson submanifold and $N^{\ast 1,0} Z$ is degenerate. The exceptional divisor $\widetilde{Z}$ is a Poisson submanifold if and only if $N^{\ast 1,0}Z$ is Abelian.
\label{14:35}
\end{prop}
\begin{proof}
Let $z^1,\ldots,z^k$ be local coordinates on $X$ with $Z=\{z^1,\ldots,z^l=0\}$ for some $l\leq k$. This is covered by $l$ charts on $\widetilde{X}$ on which the projection has the form (c.f. (\ref{11:10}))
\begin{align}
p:(v^1,\ldots,z^a,\ldots,v^l,z^{l+1},\ldots,z^k)\mapsto (z^av^1,\ldots,z^a,\ldots ,z^av^l,z^{l+1},\ldots, z^k)
\label{16:52:27}
\end{align}
for $a\leq l$. Then $p$ is an isomorphism on the open dense set $\{z^a\neq 0\}$, where we have $v^j=z^j/z^a$. We have to verify when the brackets extend smoothly over the exceptional divisor $\{z^a=0\}$. There are two types of brackets that cause trouble. Firstly,  
\begin{align}
\{z^i,v^j\}=\{z^i,\frac{z^j}{z^a}\}=\frac{1}{z^a}\{z^i,z^j\}-\frac{z^j}{(z^a)^2}\{z^i,z^a\}, 
\label{15:34}
\end{align}
for $i=a$ or $i>l$, and $j \leq l$ with $j\neq a$. Secondly,  
\begin{align}
\{v^i,v^j\}=\{\frac{z^i}{z^a},\frac{z^j}{z^a}\}=\frac{1}{(z^a)^3} \big( z^a\{z^i,z^j\}+z^i\{z^j,z^a\}+z^j\{z^a,z^i\}\big),
\label{16:08}
\end{align}
for $1\leq  i,j\leq l$, $i\neq a \neq j$. Now (\ref{15:34}) extends smoothly over $z^a=0$ for all $a$ if and only if $I_Z$ is Poisson, while (\ref{16:08}) extends over $z^a=0$ for all $a$ if and only if $I_Z$ is degenerate in the sense of (\ref{14:07:39}). Finally, $I_{\widetilde{Z}}$ is generated by $z^a$ and this is a Poisson ideal if and only if the right hand side of (\ref{15:34}) for $i=a$ is divisible by $z^a$, which is equivalent to $\{I_Z,I_Z\}\subset I_Z^2$. 
\end{proof}
\noindent If $Y\subset (M,\J)$ is a generalized Poisson submanifold then $Y$ is in particular a Poisson submanifold for $\pi_\J$, and so $N^\ast Y$ inherits a fiberwise Lie algebra structure in the same manner as discussed above in the holomorphic Poisson context. As e.g. shown in the proof below, this Lie bracket is complex linear with respect to the complex structure on $N^\ast Y$ induced by $\J$. We call $N^\ast Y$ degenerate if the Lie algebra structure is degenerate over $\C$.
\begin{thm}
Let $Y\subset (M,\J)$ be a generalized Poisson submanifold and let $p:\widetilde{M}\rightarrow M$ denote the blow-up with respect to the canonical holomorphic ideal $I_Y$. Then $\widetilde{M}$ has a generalized complex structure for which $p$ is holomorphic if and only if $N^\ast Y$ is degenerate. 
\end{thm}
\begin{proof}
Pick a local chart where $Y=W\times Z\subset (\R^{2(n-k)},\omega_0)\times (\C^k,\sigma)$ with $W$ open and $Z$ a holomorphic Poisson submanifold (c.f. the proof of Proposition \ref{11:14:11}). As explained in the discussion above, the generalized complex structure lifts to the blow-up if and only if $\sigma$ lifts to the blow-up of $Z$ in $\C^k$, which we now know to be equivalent to $N^{\ast 1,0}Z$ being degenerate. Denote by $N^\ast Z$ the normal bundle of $Z$ considered as a real submanifold, which carries a complex structure because $Z$ is a complex submanifold. If $Q=Re(\sigma)$ we have 
\begin{align*}
[\alpha,\beta]_Q=d Q(\alpha,\beta)=d(\frac{1}{2}\sigma(\alpha^{1,0},\beta^{1,0})+\frac{1}{2}\bar{\sigma}(\alpha^{0,1},\beta^{0,1}))=\frac{1}{2}[\alpha^{1,0},\beta^{1,0}]_\sigma+\frac{1}{2}[\alpha^{0,1},\beta^{0,1}]_{\bar{\sigma}},\end{align*}
for $\alpha,\beta\in N^\ast Z$. Consequently the complex isomorphism $N^\ast Z\rightarrow N^{\ast 1,0}Z$ given by $\alpha\mapsto \alpha^{1,0}$ carries $[,]_Q$ over to $\frac{1}{2}[,]_{\sigma}$. In particular, $N^{\ast 1,0}Z$ is degenerate if and only if $(N^\ast Z,[,]_Q)$ is degenerate as a complex Lie algebra. Now in the local chart $N^\ast Y=N^\ast Z$ and $\pi_\J=-\omega_0^{-1}\oplus 4IQ$. Hence $[,]_Q$ and $[,]_{\pi_\J}$ agree up to a complex multiple and so one is degenerate over $\C$ if and only if the other one is.
\end{proof}
\begin{ex}
Let $(M,\J)$ be a generalized complex manifold. In \cite{2014arXiv1403.6909B} it is shown that the complex locus, i.e. the points of type $0$, carries canonically the structure of a complex analytic space. Any complex submanifold of the complex locus is then a generalized Poisson submanifold and can be blown up as soon as its conormal bundle is degenerate. The easiest applications are in complex codimension $2$ where degeneracy is automatic. For example, any point in the complex locus on a generalized complex four-manifold can be blown up. This generalizes the corresponding result from \cite{MR2574746} where it was assumed that the point lies in the smooth part of the complex locus. An example where the submanifold has positive dimension is the maximal torus $S^1\times S^1\subset S^3\times S^3$. As is shown e.g. in \cite{gualtieri-2010}, even-dimensional reductive compact Lie groups admit \textsl{generalized K\"ahler structures}, i.e. commuting pairs of generalized complex structures $(\J_1,\J_2)$ for which $(u,v)\mapsto \langle\J_1u,\J_2v\rangle$ is positive definite on $\T M$. These are built out of left- and right-invariant complex structures on the group and the complex locus for $\J_1$ consists of those points where these two coincide. Therefore the maximal torus will be a generalized Poisson submanifold for $\J_1$. In the particular example noted above the maximal torus is of complex codimension $2$ so it is automatically degenerate. More details about this example can be found in \cite{GK}.   
\label{20:24:24}
\end{ex}

\begin{ex}
Let $(M,\J)$ be a $4$--dimensional generalized complex manifold which is generically symplectic with non-empty complex locus $Z$. Since $Z$ is locally described by the vanishing of a holomorphic Poisson tensor in two complex dimensions, it looks locally like a complex curve. By the previous example we can blow up any point on $Z$ and one can use this to ``desingularize" $Z$. Indeed, as is proven for example in \cite{MR749574}, if $C\subset X$ is any complex curve on a complex smooth surface $X$, one can perform a locally finite number of blow-ups on $X$ so that the underlying analytic set of the total transform of the curve $C$ has only ordinary double points. In particular, the total transform\footnote{The \textsl{total transform} of a subset $C$ under a blow-up equals $\pi^{-1}(C)$ where $\pi$ is the blow-down map, while the \textsl{proper transform} equals $\overline{\pi^{-1}(C)\backslash E}$. } itself will be a normal crossing divisor with possible multiplicities (so in local coordinates $z_1,z_2$ it will be given by $z_1^az_2^b=0$ for some $a,b\in \Z_{>0}$). Now we do not have a global complex structure available but this desingularization procedure is purely local, so we conclude that after a (locally finite) number of blow-ups we get a generalized complex manifold whose complex locus, as a complex analytic space, has only normal crossings. 
\end{ex}

\subsection{Generalized Poisson transversals}
\label{12:15}

We now turn our attention to submanifolds which are symplectic in transverse directions.
\begin{defn}
Let $(M,\J)$ be a generalized complex manifold. A \textsl{generalized Poisson transversal} is a submanifold $Y\subset M$ with 
\begin{align}
\J(N^\ast Y)\cap (N^\ast Y)^\perp =0.
\label{17:58:30}
\end{align}
\end{defn}
\begin{rem} 
The above condition automatically implies $N^\ast Y\cap \J N^\ast Y=0$, hence generalized Poisson transversals are in particular generalized complex submanifolds\footnote{As $i^\ast L\cap N^\ast Y_\C=0$, $\mfB i(L)$ is smooth. Here $i:Y\hookrightarrow M$ denotes the inclusion.}. Note that (\ref{17:58:30}) is equivalent to $\pi_\J(N^\ast Y)+TY=TM|_Y$, i.e. $Y$ is a Poisson transversal for the underlying Poisson structure $\pi_\J$. Geometrically, $Y$ intersects the symplectic leaves of $\pi_\J$ transversally and symplectically. Note that if $\J$ is complex then $Y$ has to be an open subset while if $\J$ is symplectic then $Y$ has to be a symplectic submanifold. 
\end{rem}

\subsubsection{A normal form}

Let $Y \hookrightarrow (M,J)$ be a generalized Poisson transversal. To blow up $Y$ we need a description of a neighborhood of $Y$ in $M$. Since $Y$ is a generalized complex submanifold it has its own generalized complex structure $\J_Y$. Moreover, the splitting $TM|_Y=TY\oplus NY$, with $NY:=\pi_\J(N^\ast Y)$, induces a decomposition $(\pi_\J)|_Y=\pi_{\J_Y}+\omega_Y$, where $\pi_{\J_Y}$ coincides with the Poisson structure on $Y$ induced by $\J_Y$ and $\omega_Y\in \Gamma(\bigwedge^2 NY)$ is non-degenerate. The suggestive notation for the latter indicates that we will consider $\omega_Y$ as a symplectic form on the bundle $N^\ast Y$. In what follows we will implicitly identify $Y$ with the zero section in $N^\ast Y$.

\begin{thm} Associated to the data $(\J_Y,\omega_Y)$ there is a natural family of mutually isotopic generalized complex structures on a neighborhood of $Y$ in $N^\ast Y$. 
\label{16:24:12}
\end{thm}

\begin{proof}
Recall that $T(N^\ast Y)$ has a canonical decomposition along $Y$ given by 
\begin{align}
T(N^\ast Y)|_Y=N^\ast Y\oplus TY.
\label{16:24:45}
\end{align}
\begin{lem}
There exists a closed $2$--form $\sigma$ on $N^\ast Y$ which along $Y$ is given by $\omega_Y\oplus 0$.
\label{13:53}
\end{lem}
\begin{proof}
Choose an Hermitian structure $(g,I)$ on $N^\ast Y$ compatible with $\omega_Y$. Let $e_j$ be a local unitary frame with dual frame $e^j$, such that $\omega_Y=\sum_j \frac{i}{2}e^j\bar{e}^j$. We obtain local coordinates $(x,z)$ on $N^\ast Y$ by identifying $(x,z)$ with the point $\sum_j z^j e_j(x)$. Note that the $z$--coordinates are complex. If $\rho_\alpha$ is a partition of unity and $e_j^\alpha$ are local frames as above, define 
\begin{align}
\lambda:=\sum_{\alpha,j}p^\ast(\rho_\alpha)\frac{i}{2}z^{\alpha j }d\bar{z}^{\alpha j}.
\label{15:58:20} 
\end{align}
Then $\sigma:=d\lambda$ restricts to $\omega_Y$ on $Y$ and its restriction to each fiber of $N^\ast Y$ is the translation invariant extension of $\omega_Y$. In addition, this particular choice of $\lambda$ is also $U(1)$-invariant. 
\end{proof}
\noindent
If $\sigma$ is a closed extension of $\omega_Y$ as above we define a Dirac structure $L_\sigma$ on $N^\ast Y$ by 
\begin{align}
L_\sigma:=e^{i\sigma}_\ast (\mfB p(L_Y)), 
\label{14:01}
\end{align} 
where $p:N^\ast Y\rightarrow Y$ is the projection. It is integrable with respect to the $3$--form $\widetilde{H}:=p^\ast H_Y$ where $H_Y:= i^\ast H$, and along the zero section we have 
\[L_\sigma |_{{}_Y}=\{X+\xi+e-i\omega_Y(e)|X+\xi\in L_Y,e\in N^\ast Y\},\]
where we used the decomposition (\ref{16:24:45}). In particular $L_\sigma\cap \overline{L_\sigma}=0$ at $Y$, hence also in a neighborhood of $Y$ in $N^\ast Y$. We will denote the resulting generalized complex structure by $\J_\sigma$. The family of the theorem is by definition the set of $\J_\sigma$, where $\sigma$ ranges over the closed extensions of $\omega_Y$.
\begin{lem}
Let $\sigma_t$ be a family of closed $2$-forms extending $\omega_Y$. Then there exists a family $\Phi_t=(\varphi_t,B_t)$ of diffeomorphisms around $Y$ with $\Phi_0=(Id,0)$ that satisfies $\mfF \Phi_t (L_0)=L_t$ and which fixes $Y$ up to first order, i.e.  $\varphi_t|_Y=Id$, $d\varphi_t|_Y=Id$ and $B_t|_Y=0$.
\label{20:58:39}
\end{lem}
\begin{proof}
Since $\sigma_t-\sigma_0$ vanishes on $Y$, Lemma \ref{17:04:06} provides a family $\eta_t\in\Omega^1(N^\ast Y)$ with $\sigma_t-\sigma_0=d\eta_t$ and such that $\eta_t$ and its first partial derivatives vanish along $Y$. By definition,
\begin{align*}
L_t:=L_{\sigma_t}=e^{i\sigma_t}_\ast(\mfB p(L_Y))=e^{id\eta_t}_\ast(L_{\sigma_0}).
\end{align*}
Since $\eta_t$ and $d\eta_t$ vanish along $Y$, $L_t$ defines a family of generalized complex structures $\J_t$ in a neighborhood of $Y$, integrable with respect to the (fixed) $3$--form $\widetilde{H}$. Consider the time-dependent generalized vector field $\J_t\dot{\eta}_t=:X_t+\xi_t$ and let $\psi_{t,s}$ be its flow, given by 
\begin{align}
\psi_{t,s}=(\varphi_{t,s})_\ast \circ e_\ast^{-\int_s^t\varphi_{r,s}^\ast(d\xi_r+\iota_{X_r}\widetilde{H})dr}
\end{align}
where $\varphi_{t,s}$ is the flow of the time-dependent vector field $X_t$. Since $\eta_t$ together with its first derivatives vanish along $Y$, $\varphi_{t,s}$ is well defined in a neighborhood of $Y$ and fixes $Y$ to first order. We claim that 
\begin{align}
L_t=\psi_{t,0}L_0.
\label{10:38:04}
\end{align}
From the formula for $L_t$ this amounts to showing that $e^{-id\eta_t}_\ast\psi_{t,0}L_0=L_0$. We have
\begin{align}
\frac{d}{dt}e^{-id\eta_t}_\ast\psi_{t,0}(u)&=-i\lb \dot{\eta}_t, e^{-id\eta_t}_\ast\psi_{t,0}(u)\rb-e^{-id\eta_t}_\ast \lb \J_t\dot{\eta}_t,\psi_{t,0} (u)\rb\nonumber\\ &= \lb -i\dot{\eta}_t-\J_0\dot{\eta}_t, e^{-id\eta_t}_\ast\psi_{t,0}(u) \rb 
\label{13:37:52}
\end{align}
This shows that $e^{-id\eta_t}_\ast\psi_{t,0}$ integrates the adjoint action of $-i\dot{\eta}_t-\J_0\dot{\eta}_t\in\Gamma(L_0)$. Since $\Gamma(L_0)$ is involutive, (\ref{10:38:04}) indeed holds. The desired family is then given by 
\begin{align*}
\Phi_t=(\varphi_t,B_t):=(\varphi_{t,0},\int_0^t \varphi_{r,0}^\ast(d\xi_r+\iota_{{}_{X_r}}H)dr).
\end{align*}
\end{proof}
\noindent Applying this lemma to $\sigma_t:=(1-t)\sigma+t\sigma'$, where $\sigma$ and $\sigma'$ are closed extensions of $\omega_Y$, shows that indeed all members of the family are mutually isotopic. This finishes the proof of Theorem \ref{16:24:12}. 
\end{proof}

\begin{lem}
Let $\alpha_t\in\Omega_{cl}^k(E)$ be a family of closed forms on a vector bundle $E$ over $M$ which vanish along $M$. Then there exists a family $\eta_t\in\Omega^{k-1}(E)$ with $d\eta_t=\alpha_t$, such that for each $t$ the form $\eta_t$ together with its first partial derivatives vanishes along $M$.  
\label{17:04:06}
\end{lem}
\begin{proof}
Let $V$ denote the Euler vector field on $E$, i.e. $V_\xi=\xi$ for $\xi\in E$. Its flow is given by $\varphi_s(\xi)=e^s\xi$, and we have 
\begin{align*}
\alpha_t=\underset{s\rightarrow -\infty}{\text{lim}} \big( \varphi_0^\ast \alpha_t-\varphi_s^\ast \alpha_t\big)
=\int_{-\infty}^0 \frac{d}{ds}\varphi_s^\ast \alpha_t ds=d\Bigg(\iota_{{}_V}\int_{-\infty}^0\varphi_s^\ast \alpha_t ds\Bigg)=:d\eta_t. 
\end{align*}
Another formula for $\eta_t$ is given by $\eta_t=\iota_{{}_V}\int_0^1 \frac{1}{s}L_s^\ast \alpha_tds$, where $L_s$ denotes left-multiplication by $s$ on $E$. The forms $\eta_t$ then satisfy all the properties of the lemma. 
\end{proof}
\noindent Theorem \ref{16:24:12} shows that any symplectic vector bundle over a generalized complex manifold has a generalized complex structure for which the base is a generalized Poisson transversal. The following theorem shows that all generalized Poisson transversals locally arise from this construction. 
\begin{thm}
There is a natural family of mutually isotopic embeddings from a neighborhood of $Y$ in $N^\ast Y$ to a neighborhood of $Y$ in $M$, which pull back $\J$ to one of the structures constructed in Theorem \ref{16:24:12}. \label{21:48:41}
\end{thm}
\begin{proof}
Let $p:T^\ast M\rightarrow M$ be the cotangent bundle and choose an arbitrary connection  $\nabla$ on $TM$, whose dual connection on $T^\ast M$ we also denote by $\nabla$. Using the Poisson structure $\pi_\J$ we obtain a vector field $V$ on $T^\ast M$, whose value at $\xi\in T^\ast M$ is given by $V_\xi:=\pi_\J(\xi)^h_\xi$. We denote by $\varphi_t:T^\ast M\rightarrow T^\ast M$ its flow. 
\begin{lem}
The map $exp:=p\circ \varphi_1|_{{}_{N^\ast Y}}:N^\ast Y \partto M$ gives a diffeomorphism from a neighborhood of $Y$ in $N^\ast Y$ onto an open neighborhood of $Y$ in $M$. If $\nabla'$ is a different connection then $exp'$ is isotopic to $exp$ via maps which are constant on $Y$ up to first order. 
\label{17:14:20}
\end{lem}
\begin{proof}
By definition of $V$ we have $L_s^\ast V=sV$ for $s\in \R$, where $L_s$ denotes multiplication by $s$ on the fibers of $T^\ast M$. It follows that\footnote{This equality is similar to the more familiar equality $\gamma_{{}_{sX}}(t)=\gamma_{{}_X}(st)$ for geodesics.} $\varphi_t(L_s\xi)=L_s(\varphi_{st}(\xi))$ for $\xi\in T^\ast M$. Hence,  
\begin{align*}
d_y\varphi_t(\xi)=\left. \frac{d}{ds} \right |_{s=0} \varphi_{t}(L_s\xi)=\xi+t\pi_\J(\xi),
\end{align*}
for $y\in Y\subset N^\ast Y$. Since $V$ vanishes at $Y$ we have $exp|_{{}_Y}=Id$ and so
\begin{align}
d_y\varphi_t(\xi,v)=(\xi,v+t\pi_\J(\xi))
\label{21:45:59}
\end{align}
in terms of the decomposition (\ref{16:24:45}). Composing with $p$ gives 
$
d_{y}exp(\xi,v)=v+\pi_\J(\xi),
$ hence by transversality of $Y$ we see that $exp$ is a local diffeomorphism. Since $exp|_{{}_Y}=Id$ and $Y$ is properly embedded, $exp$ is a diffeomorphism around $Y$. If $\nabla'$ is a different connection there is a path of connections $\nabla^t$ from $\nabla$ to $\nabla'$, whose exponentials $exp^t$ give an isotopy. Since (\ref{21:45:59}) is independent of $\nabla^t$, the $exp^t$ all agree up to first order at $Y$.
\end{proof} 
\noindent We will now construct explicitly one of the generalized complex structures from Theorem \ref{16:24:12} together with a $2$-form $B$ on $N^\ast Y$ such that $(exp,B)$ is holomorphic. For the proof of the following lemma recall that for $X,Y\in TM$, $\alpha, \beta\in T^\ast M$, 
\begin{align}
(\omega_{can})_{y}(\alpha+X,\beta+Y)=\alpha(Y)-\beta(X)
\label{09:51:16}
\end{align}
in terms of $T(T^\ast M)|_{{}_M}=T^\ast M\oplus TM$. 
\begin{lem}
Define  
\begin{align}
\widetilde{\sigma}_t:=-\int_0^t(\varphi_s)^\ast \omega_{can}ds \ \in \ \Omega^2_{cl}(T^\ast M),
\label{15:12:03}
\end{align} 
where $\varphi_s$ is the flow of the vector field $V$. Then $\sigma:=i^\ast\widetilde{\sigma}_1$ is a closed extension of $\omega_Y$, where $i:N^\ast Y\hookrightarrow T^\ast M$ denotes the inclusion. 
\label{17:10:48}
\end{lem}
\begin{proof}
Using (\ref{21:45:59}) and (\ref{09:51:16}) we see that 
\begin{align*}
\sigma_{{}_{y}}(\alpha+X,\beta+Y)=&-\int_0^1 (\omega_{can})_{y}\big(\alpha+(X+s\pi_\J(\alpha)),\beta+(Y+s\pi_\J(\beta))\big)	ds\\
=&-\int_0^1 2s \alpha(\pi_\J(\beta))ds
=\omega_Y(\alpha,\beta)
\end{align*}
for all $\alpha,\beta\in N^\ast Y$ and $X,Y\in TY$, proving the lemma. 
\end{proof}
\noindent
The vector field $V$ is part of the generalized vector field $\mcV$ on $T^\ast M$ where $\mcV_\xi:=(\J\xi)_\xi^h$. If $\psi_t$ denotes the flow of $\mcV$ on $\T (T^\ast M)$ then a computation similar to (\ref{13:37:52}) shows that $\psi_t e^{i\widetilde{\sigma}_t}_\ast$ integrates the adjoint action of $-i\lambda_{can}-\mcV$. Since $(-i\lambda_{can}-\mcV)_\xi=(-i\xi-\J\xi)_\xi^h\in \mfB p(L)$ and $\mfB p(L)$ is involutive, $\psi_t e^{i\widetilde{\sigma}_t}_\ast$ preserves $\mfB p(L)$ and so   
\begin{align}
e^{i\widetilde{\sigma}_t}_\ast \mfB p(L)= \psi_{-t}\mfB p(L)
\label{14:49}
\end{align} 
as Dirac structures on $T^\ast M$. Here is an overview of all the maps involved:
\begin{align*}
\xymatrix{
N^\ast Y \ar[d]^p \ar[r]^i &T^\ast M \ar[d]^p \ar[r]^{\varphi_t} & T^\ast M \ar[ld]^p \\
Y \ar[r]          &M}
\end{align*}
The left square is commutative but the right triangle is not. Now if we apply $\mfB i$ to (\ref{14:49}) at $t=1$, the left hand side becomes $e^{i\sigma}_\ast \mfB i \mfB p(L)=e^{i\sigma}_\ast\mfB p(L_Y)$ where $\sigma=i^\ast \widetilde{\sigma}_1$. This is precisely one of the structures from Theorem \ref{16:24:12}. If we write $\psi_t=(\varphi_t)_\ast e^{-B_t}_\ast$ (see (\ref{10:12:54})), the right hand side becomes
\begin{align*}
\mfB i (\psi_{-1} \mfB p(L))=\mfB i \mfB\Phi_1 \mfB p (L)=\mfB(p\circ \Phi_1 \circ i) (L)
\end{align*}
where $\Phi_t:=(\varphi_t,B_t)$. Now $p\circ \Phi_1\circ i=(exp, i^\ast B_1)$, so if we define $B:=i^\ast B_1$ then $(exp,B)$ is indeed holomorphic. Note that both $exp$ and $B$ only depend on $\nabla$ and $\J$. If $\nabla'$ is a different connection we choose a path of connections $\nabla^t$ connecting them, giving rise to a family of embeddings $(exp^t,B^t)$. As in Lemma \ref{17:14:20} one can show that this family fixes $Y$ up to first order. This completes the proof of Theorem \ref{21:48:41}. \end{proof}

\subsubsection{Blowing up}

In this section we will use the normal form theorem for $Y$ to construct the symplectic version of the blow-up. To motivate the upcoming discussion let us recall how to blow up a point using symplectic cuts (cf. \cite{MR1338784}). Let $\omega_{st}=\frac{i}{2}(dwd\bar{w}+dzd\bar{z})$ be the standard symplectic structure on $\C\times \C^n$ and consider the Hamiltonian $S^1$-action given by $e^{i\theta}\cdot(w,z)=(e^{i\theta}w,e^{-i\theta}z)$, with moment map 
\begin{align}
\mu(w,z)=\frac{1}{2}(|z|^2-|w|^2).
\label{14:18}
\end{align}
Now $S^1$ acts freely on $\mu^{-1}(\frac{1}{2}\epsilon^2)$ for $\epsilon>0$ and the map $\kappa:\mu^{-1}(\frac{1}{2}\epsilon^2)\rightarrow \C^n\times \mbP^{n-1}$ given by
\begin{align*}
\kappa:(w,z)\mapsto (\frac{wz}{|z|},[z])
\end{align*}
induces a diffeomorphism from $\mu^{-1}(\frac{1}{2}\epsilon^2)/S^1$ onto $\widetilde{\C}^n=\{(x,l)|x\in l\}$, the blow-up of $\C^n$ at the origin. It is a well-known fact that $\kappa^\ast (pr_1^\ast \omega_{st}+\epsilon^2pr_2^\ast\omega_{FS})=\omega_{st}$, giving an explicit description of the symplectic form on the reduced space. Finally, consider the following slice for the $S^1$-action: 
\begin{align*}
\varphi:\C^n\backslash \overline{B_\epsilon} \rightarrow \mu^{-1}(\frac{1}{2}\epsilon^2), \ \ \ \ \ u &\mapsto (\sqrt{|u|^2-\epsilon^2},u).
\end{align*}
Here $B_\epsilon$ is the ball of radius $\epsilon$. Clearly $\varphi^\ast \omega_{st}=\frac{i}{2}dud\bar{u}$, which shows that the symplectic quotient $\mu^{-1}(\frac{1}{2}\epsilon^2)/S^1$ is symplectomorphic, away from the exceptional divisor, to $(\C^n\backslash \overline{B_\epsilon},\omega_{st})$. 
\newline
\newline
To use this in our setting we need a reduction procedure for generalized complex structures. A general reduction theory has been introduced in \cite{MR2323543}, but we only need a very special case which we will present here. In what follows, an $S^1$-action on $(Z,H,\J)$ is understood to be an $S^1$-action on the manifold $Z$ which preserves $\J$ and for which $\iota_XH=0$, where $X$ is the associated action vector field. In analogy with symplectic geometry we call $\mu:Z\rightarrow \R$ a \textsl{moment map} if $\J X=d\mu$. 
\begin{prop}
\label{17:31:26}
Suppose we have an $S^1$-action on $(Z,H,\J)$ with moment map $\mu$. If $i:\mu^{-1}(c)\hookrightarrow Z$ is a regular level set with quotient $q: \mu^{-1}(c)\rightarrow \mu^{-1}(c)/S^1$, then $\mfF q(\mfB i(L))$ gives a generalized complex structure $\J'$ on $\mu^{-1}(c)/S^1$. If $\rho$ is a local spinor for $\J$ which is $S^1$-invariant, then $i^\ast \rho=q^\ast \rho'$ for a unique form $\rho'$ on the quotient which is a spinor for $\J'$. 
\end{prop}
\begin{proof}
The inclusion of a regular level set $i:\mu^{-1}(c)\hookrightarrow Z$ has real codimension $1$ so that $\mfB i(L)$ is automatically smooth, and we have
\begin{align}
\mfB i(L)\cap \mfB i(\bar{L})=\C\cdot X.
\label{12:42:33}
\end{align}   
By the assumption $\iota_XH=0$ we can write $H=q^\ast H'$ for a (unique) $3$--form $H'$ on the quotient, so $q$ is a generalized map. It satisfies $ker(dq)\cap \mfB i(L)=\C\cdot X$, which is of constant rank $1$ so the forward image $\mfF q(\mfB i(L))$ is smooth, and projects down to $\mu^{-1}(c)/S^1$ because $\mfB i(L)$ is $S^1$-invariant. It is generalized complex because of (\ref{12:42:33}) and the fact that $X$ spans the kernel of $q_\ast$. Let $\rho$ be a local spinor for $L$ which is $S^1$-invariant. Then $i^\ast \rho$ is nonzero on $\mu^{-1}(c)$ and is an $S^1$-invariant spinor for $\mfB i(L)$. Moreover, 
\begin{align*}
0=(X-i\J X)\cdot \rho=(X-id\mu)\cdot \rho
\end{align*}
implies that $\iota_{{}_X} i^\ast \rho=0$, hence $i^\ast \rho$ comes from a unique differential form on $\mu^{-1}(c)/S^1$. This will be a spinor for the induced generalized complex structure on the quotient.
\end{proof}
\noindent Consider now a generalized Poisson transversal $Y\subset (M,\J)$, with $\omega_Y$ the induced symplectic structure on $N^\ast Y$. As in the proof of Lemma \ref{13:53} we choose a compatible Hermitian structure $(g,I)$ on the bundle $N^\ast Y$ and use it to construct an $S^1$-invariant $1$--form $\lambda$ on the manifold $N^\ast Y$ of the form (\ref{15:58:20}). In particular its differential $\sigma=d\lambda$ is a closed extension of $\omega_Y$ which is $S^1$-invariant and whose restriction to the fibers is translation invariant. Consider the $S^1$-action on $Z:=\C\times N^\ast Y$ given by 
\begin{align*}
e^{i\theta}\cdot (w,z)=(e^{i\theta}w,e^{-i\theta}z),
\end{align*} 
and denote by $X\in \Gamma(TZ)$ the induced action vector field. We equip $Z$ with the $3$--form $p^\ast H_Y$ and the generalized complex structure which is the product of the standard symplectic structure on $\C$ and $\J_\sigma$ on $N^\ast Y$ as defined by Equation (\ref{14:01}). 
\begin{lem}
The map $\mu:Z\rightarrow \R$ given by $\mu(w,z):=\frac{1}{2}g(z,z)-\frac{1}{2}|w|^2$ is a moment map.
\end{lem}
\begin{proof}
We can write $X=(X_1,X_2)$ on $\C\times N^\ast Y$ with $X_i$ the corresponding action vector field on the separate factors. In particular $X_2$ is vertical and by definition of $\J_\sigma$ we have $\J(X_1,X_2)=(\omega_{st}(X_1),\sigma(X_2))$. Since $\omega_{st}+\sigma=d(\lambda_{st}+\lambda)$ where both $\lambda_{st}$ and $\lambda$ are $S^1$-invariant, we get $\J X=-d\iota_X(\lambda_{st}+\lambda)$. Hence it suffices to show that $-\iota_X(\lambda_{st}+\lambda)=\mu$. This is a fiberwise equality and can be verified on $\C\times \C^n$.  
\end{proof}
\begin{rem}
If one starts with an arbitrary extension $\sigma=d\lambda$ of $\omega_Y$ one can average it over $S^1$ to render it invariant, and the map $-\iota_X(\lambda_{st}+\lambda)$ is again a moment map. The advantage of our choice above is that the moment map has an explicit description in terms of a metric.  
\end{rem}
\noindent For $\epsilon>0$, Proposition \ref{17:31:26} implies that $\widetilde{N^\ast Y}_\epsilon:=\mu^{-1}(\frac{1}{2}\epsilon^2)/S^1$ is generalized complex, which differentiably equals the blow-up of $Y$ in $N^\ast Y$. It remains to show that this blow-up can be glued back into the original manifold $M$ to produce the blow-up of $Y$ in $M$. For that we consider the slice 
\begin{align}
\widetilde{\varphi} :  N^\ast Y\backslash \overline{B_\epsilon}\hookrightarrow \mu^{-1}(\frac{1}{2}\epsilon^2)\subset Z, \ \ \ \ \  z\mapsto (\sqrt{|z|^2-\epsilon^2},z).
\label{18:00:10}
\end{align} 
Here $\overline{B_\epsilon}$ is the disc bundle of radius $\epsilon$. If $q$ denotes the quotient map of the $S^1$-action, we obtain a diffeomorphism 
\begin{align*}
\varphi:=q\circ \widetilde{\varphi}:N^\ast Y\backslash\overline{B_\epsilon}\longrightarrow \widetilde{N^\ast Y}_\epsilon\backslash E
\end{align*}
where $E$ denotes the exceptional divisor. To show that $\varphi$ is holomorphic it suffices, by definition of the generalized complex structure on the quotient, to show that $\widetilde{\varphi}$ pulls back a local spinor on $Z$ to a local spinor for $\J_\sigma$. If $\rho=e^{i\omega_{st}+i\sigma}\wedge p^\ast \rho_Y$ is such a spinor on $Z$, then from the definition of $\widetilde{\varphi}$ we see that indeed $\widetilde{\varphi}^\ast \rho=e^{i\sigma}\wedge p^\ast \rho_Y$ is a spinor for $\J_\sigma$.

\begin{thm}
\label{20:28:30}
Let $Y\subset (M,\J)$ be a compact generalized Poisson transversal. Then the differentiable blow-up of $Y$ in $M$ carries a generalized complex structure which, outside of a neighborhood of the exceptional divisor, is isomorphic to the complement of a neighborhood of $Y$ in $M$. The result is, up to deformation, independent of choices. 
\end{thm}
\begin{proof}
Equip a neighborhood $U$ of $Y$ in $N^\ast Y$ with the generalized complex structure $J_\sigma$ where $\sigma$ is as above. By Theorems \ref{16:24:12} and \ref{21:48:41}, if $U$ is small enough there is a holomorphic embedding $\iota:(U,\J_\sigma)\rightarrow (\iota(U),\J)$ with $\iota(U)$ a neighborhood of $Y$ in $M$. Since $Y$ is compact there is an $\epsilon>0$ such that $\overline{B_\epsilon}\subset U$. Set $\widetilde{U}:=\varphi(U)\cup E$ and define the blow-up of $Y$ in $M$ by
\begin{align}
\label{14:57:13}
\widetilde{M}:=M\backslash \iota(\overline{B_\epsilon}) \underset{\iota\circ \varphi^{-1}}{\cup} \widetilde{U}.
\end{align}
Here the glueing takes place between $\widetilde{U}\backslash E$ and $\iota(U\backslash \overline{B_\epsilon})$. Any two different choices of generalized complex structures on $N^\ast Y$ from Theorem \ref{16:24:12} or embeddings into $M$ from Theorem \ref{21:48:41} are isotopic to each other, hence by performing the above construction in families we see that the resulting blow-ups are deformation equivalent. 
\end{proof}
\begin{rem}
$i).$ The drawback of defining $\widetilde{M}$ by (\ref{14:57:13}) is that there is no canonical blow-down map. It is certainly possible to define blow-down maps to $M$, but they are not particularly useful because they will not be holomorphic around the exceptional divisor.   
\newline
$ii).$ If $\J$ is symplectic around $Y$ then the resulting structure on the blow-up is actually independent of choices up to isotopy. This follows because the family of forms all represent the same cohomology class, as the divisors all carry the fixed prescribed symplectic area. It is unclear to the authors if a similar statement for Theorem \ref{20:28:30} holds as well.
\end{rem}

\begin{ex}
Let $(M,\J_1,\J_2)$ be a generalized K\"ahler manifold and $Y\hookrightarrow M$ a generalized Poisson submanifold for $\J_1$, i.e. $\J_1N^\ast Y=N^\ast Y$. Since $\langle \J_1\alpha, \J_2\alpha\rangle >0$ for all $\alpha\in N^\ast Y$, we see that $\J_2N^\ast Y\cap (N^\ast Y)^\perp=0$, i.e. $Y$ is a generalized Poisson transversal\footnote{In K\"ahler geometry this amounts to the well-known fact that a complex submanifold is automatically symplectic.} for $J_2$. In Example \ref{20:24:24} we discussed how the maximal torus in a compact even dimensional Lie group is a generalized Poisson submanifold for $\J_1$ which, because of the degeneracy condition, can almost never be blown up. With respect to $\J_2$ however there are no further restrictions, so all maximal tori can be blown up for $\J_2$. In \cite{GK} a more thorough investigation of these examples will be given and it will be shown that, if the maximal torus can be blown up for $\J_1$ and $\J_2$, then the result is again generalized K\"ahler. 
\end{ex}

\subsection{A remark on other types of submanifolds}
\label{11:46:45}

Our definition of a generalized complex submanifold is, besides a smoothness criterion, characterized by 
\begin{align}
\label{09:07:51}
JN^\ast Y\cap (N^\ast Y)^\perp \subset N^\ast Y. 
\end{align}
In the previous sections we investigated the blow-up theory of the two extreme cases, namely those for which the above inclusion is either an equality (the generalized Poisson case) or the intersection is zero (the generalized Poisson transversals). An obvious question at this point is whether the ``intermediate" cases admit a blow-up theory as well. The techniques we used for the generalized Poisson submanifolds and the generalized Poisson transversals are so different from each other that it does not seem we can use either of them when the type in the normal direction is mixed. We will now give an example where we can explicitly prove that there does not exist a blow-up. For that we will use the following 
\begin{prop}
Let $M$ be a compact $4$--dimensional generalized complex manifold of type $1$. Then the Euler characteristic $\chi(M)$ is even. 
\end{prop}
\begin{proof}
A type $1$ structure gives rise to a decomposition $TM=L_1\oplus L_2$ where $L_1$ is the distribution tangent to the symplectic foliation and $L_2$ is a choice of normal bundle. In particular $L_1$ and $L_2$ are orientable and we can think of them as complex line bundles\footnote{In fact $L_2$ inherits a canonical almost complex structure, being the normal to the symplectic foliation in a generalized complex manifold.}, giving an almost complex structure on $TM$. By Wu's formula, using that $c_1(TM)\equiv w_2(M)$ mod $2$ and $c_1(TM)=c_1(L_1)+c_1(L_2)$, we obtain 
\begin{align*}
\alpha^2\equiv \alpha\cup c_1(L_1)+\alpha\cup c_1(L_2) \ \text{mod} \  2 \  \ \ \forall \alpha\in H^2(M,\mathbb{Z}).
\end{align*}
 Applying this to $\alpha=c_1(L_1)$ we see indeed that $\chi(M)=c_1(L_1)c_1(L_2)$ is even. 
\end{proof}
\noindent Now let $M$ be a compact $4$--dimensional generalized complex manifold of type $1$. The blow-up of a point in $M$ is differentiably given by $M\#\overline{\C\mbP}^2$, which has Euler characteristic $\chi(M)+1$. If the blow-up would have a generalized complex structure that agrees with the one on $M$ outside a neighborhood of the exceptional divisor, it would have type $1$ everywhere since the type can only change in even amounts. By the Proposition we conclude that the blow-up can not be generalized complex, at least not in a way that is reasonably related to the original structure on $M$. 

In the example above, Equation (\ref{09:07:51}) is neither zero nor an equality. There are however, generalized complex submanifolds $Y$ for which (\ref{09:07:51}) is zero at some points and an equality at others. Further study is needed to see what can be said about these types of submanifolds. 




\bibliographystyle{hyperamsplain}
\bibliography{references}

\end{document}